\documentclass[12pt,a4paper]{amsart}

\usepackage[OT1]{fontenc}    %
\usepackage{type1cm}         %
\usepackage{amsthm}          %
\usepackage{amsmath}
\usepackage{amssymb}
\usepackage[dvips]{graphicx} 
\usepackage[bf]{caption}     
\usepackage{psfrag}          
\usepackage[active]{srcltx}  
\usepackage{geometry}        
\usepackage{version}         

\RequirePackage[dvips]{graphicx}

\numberwithin{equation}{section}

\theoremstyle{plain}
\newtheorem{theorem}{Theorem}[section]

\newtheorem{proposition}[theorem]{Proposition}
\newtheorem{lemma}[theorem]{Lemma}

\newtheorem{conjecture}[theorem]{Conjecture}

\theoremstyle{remark}

\theoremstyle{definition}

\newcommand{\R}{\mathbb{R}}

\newcommand{\Q}{\mathbb{Q}}
\newcommand{\Z}{\mathbb{Z}}
\newcommand{\N}{\mathbb{N}}
\newcommand{\cP}{\mathcal{P}}

\newcommand{\ldim}{\underline{\dim}}
\newcommand{\udim}{\overline{\dim}}
\newcommand{\F}{\mathcal{F}}
\newcommand{\e}{\varepsilon}

\DeclareMathOperator{\Dim}{Dim}
\DeclareMathOperator{\supp}{supp}

  \newcommand{\hdim}{\dim_{\mathsf{H}}}
  \newcommand{\fdim}{\dim_{\mathsf{F}}}
  \newcommand{\pdim}{\dim_{\mathsf{P}}}
  \newcommand{\ubdim}{\overline{\dim}_{\mathsf{B}}}
  \newcommand{\sdim}{\dim_{\mathsf{S}}}

\begin{document}

\title[Projections of self-similar and other fractals]{Projections of self-similar and related fractals: a survey of recent developments}

\author{Pablo Shmerkin}
\address{Department of Mathematics and Statistics, Torcuato Di Tella University, and CONICET, Buenos Aires, Argentina}
\email{pshmerkin@utdt.edu}
\thanks{P.S. was partially supported by Projects PICT 2011-0436 and PICT 2013-1393 (ANPCyT)}
\urladdr{http://www.utdt.edu/profesores/pshmerkin}



\begin{abstract}
In recent years there has been much interest -and progress- in understanding projections of many concrete fractals sets and measures. The general goal is to be able to go beyond general results such as Marstrand's Theorem, and quantify the size of \textit{every} projection - or at least every projection outside some very small set. This article surveys some of these results and the techniques that were developed to obtain them, focusing on linear projections of planar self-similar sets and measures.
\end{abstract}

\maketitle

\section{Introduction}

The study of the relationship between the Hausdorff dimension of a set and that of its linear projections has a long history, dating back to Marstrand's seminal projection theorem \cite{Marstrand54}:

\begin{theorem} \label{thm:Marstrand}
Let $A\subset \mathbb{R}^2$ be a Borel set. Let $\Pi_\alpha$ denote the orthogonal projection onto a line making an angle $\alpha$ with the $x$-axis.
\begin{enumerate}
\item[(i)] If $\hdim A\le 1$, then $\hdim \Pi_\alpha A=\hdim A$ for almost every $\alpha\in [0,\pi)$.
\item[(ii)] If $\hdim A>1$, then $\mathcal{L}(\Pi_\alpha A)>0$ for almost every $\alpha\in [0,\pi)$.
\end{enumerate}
\end{theorem}
Here $\hdim$ stands for Hausdorff dimension and $\mathcal{L}$ for one-dimensional Lebesgue measure. Note that, in particular,
\[
\hdim(\Pi_\alpha A)=\min(\hdim A,1) \quad\text{for almost all } \alpha\in [0,\pi).
\]

Although Marstrand's Theorem is very general, unfortunately it does not give information about what happens for a \emph{specific} projection, and although the exceptional set is neglibible in the sense of Lebesgue measure, it may still have large Hausdorff dimension. A more recent, and very active, line of research is concerned with gaining a better understanding of the size of projections of sets with some dynamical or arithmetic structure. The goal of this article is to present an overview of this area, focusing on projections of planar self-similar sets and measures (projections of other fractals are briefly discussed in Section \ref{sec:further-results}). For a wider view of the many ramifications of Marstrand's Projection Theorem, the reader is referred to the excellent survey \cite{FFJ14}.

\section{Self-similar sets and their projections}

We review some standard terminology and fix notation along the way. An \textbf{iterated function system} (IFS) on $\R^d$ is a finite collection $\F=( f_i)_{i\in\Lambda}$ of strictly contractive self-maps of $\R^d$. As is well known, for any such IFS $\F$, there exists a unique nonempty compact set $A$ (the \textbf{attractor} or \textbf{invariant set} of $\F$) such that $A=\cup_{i\in\Lambda} f_i(A)$. We will repeatedly make use of the iterated IFS
\[
\F^k = (f_i)_{i\in\Lambda^k},  \quad\text{where }  f_{i_1\cdots i_k} = f_{i_1}\circ\cdots\circ f_{i_k},
\]
which has the same attractor as $\F$. When the maps $f_i$ are similarities, the set $A$ is a \textbf{self-similar set}. From now on the maps $f_i$ will always be assumed to be similarities, unless otherwise noted. For further background on self-similar sets and fractal dimensions, see e.g. \cite{Falconer14}.

Although we will be concerned with projections of self-similar sets, it will be useful to recall some ideas that apply to self-similar sets themselves. The \textbf{similarity dimension} $\sdim(\F)$ of an IFS $\F=(f_i)_{i\in\Lambda}$ is the only positive root $s$ of $\sum_{i\in\Lambda} \rho(f_i)^s=1$, where $\rho(f)$ is the contraction ratio of the similarity $f$. If $A$ is the attractor of $\F$, then there is a natural family of covers of $A$, namely
\[
\{ f_i(A): i\in \Lambda^k\}.
\]
Using these families, one can easily check that $\hdim A\le \sdim(A)$ (we follow a standard abuse of notation and speak of the similarity dimension of a self-similar set whenever the generating IFS is clear from context). Intuitively, it appears that if the sets $f_i(A), i\in\Lambda$ do not overlap much, then these covers should be close to optimal, and one should have an equality $\hdim A=\sdim A$. Recall that the IFS $\F$ satisfies the \textbf{open set condition (OSC)} if there exists a nonempty open set $O$ such that $f_i O\subset O$ for all $i$, and the images $f_i(O)$ are pairwise disjoint. The open set condition ensures that the overlap between the pieces $f_i(A)$ is negligible in a certain sense, and it is well known that Hausdorff and similarity dimensions agree whenever it holds. On the other hand, there are two trivial mechanisms that force the Hausdorff dimension to drop below the similarity dimension:
\begin{enumerate}
\item If $A\subset\R^d$ and $\sdim(A)>d$, then certainly $\hdim(A)\le d<\sdim(A)$.
\item If $f_i=f_j$ for some $i\neq j$, then one can drop $f_j$ from the IFS, resulting in a new generating IFS with strictly smaller similarity dimension. The same happens if two maps of $\F^k$ agree for some $k$, and in turn this happens if and only the semigroup generated by the $f_i$ is not free. In this case we say that $\F$ has an \textbf{exact overlap}.
\end{enumerate}
When the open set condition fails, but there are no exact overlaps, the combinatorial structure of the overlaps is very intricate, and calculating the dimension becomes much more challenging. In dimension $d=1$, a major conjecture in the field is whether these are the \emph{only} possible mechanisms for a drop in the Hausdorff dimension of a self-similar set (in higher dimensions this is false, but there is an analogous, albeit more complicated, conjecture).

We now turn our attention to \emph{projections} of self-similar sets. Let $A$ be a self-similar set generated by $(f_i)_{i\in\Lambda}$. If the similarities $f_i$ are homotheties, i.e. $f_i(x)=\lambda_i x+t_i$ for some contractions $\lambda_i\in (0,1)$ and translations $t_i\in\R^d$, then for any linear map $\Pi: \R^d\to\R^k$, the image $\Pi A$ is also self-similar: it is the attractor of $(\lambda_i x+\Pi t_i)_{i\in\Lambda}$. We note that even if the original self-similar set satisfies the open set condition, their projections need not satisfy it; some of them (albeit only countably many) may have exact overlaps. In general, linear projections of self-similar sets need not be self-similar.

From now on we settle on the case $d=2$, $k=1$.  We will say that a planar IFS $\F$ (or its attractor) is of \textbf{irrational type} if, for some $k$, $\F^k$ contains a map of the form $\lambda R_{\theta}x+t$ with $\theta/\pi$ irrational, where $R_\theta$ is rotation by $\theta$. Otherwise, we say that the IFS is of \textbf{rational type}.  We also say that $\F$ is \textbf{algebraic} if, when representing $f_i(x)=S_i x+t_i$ for a matrix $S_i\in\R^{2\times 2}$ and $t_i\in\R^2$, all the entries of $S_i, t_i$ are algebraic for all $i\in\Lambda$. The following theorem summarizes the current knowledge about the projections of planar self-similar sets.
\begin{theorem}  \label{thm:prof-planar-sss}
 Let $A$ be a planar self-similar set.
 \begin{enumerate}
  \item[(i)] If $A$ is of irrational type, then $\hdim \Pi_\alpha A=\min(\hdim A,1)$ for all $\alpha$.
  \item[(ii)] If $A$ is algebraic, then $\{\alpha:\hdim \Pi_\alpha A< \min(\hdim A,1)\}$ is countable.
  \item[(iii)]  $\hdim \Pi_\alpha A=\min(\hdim A,1)$ for all $\alpha$ outside of a set of zero Hausdorff (and even packing) dimension.
  \item[(iv)] If $\hdim A>1$, then $\hdim\{\alpha:\mathcal{L}(\pi_\alpha A)=0\}=0$.
 \end{enumerate}
\end{theorem}
Hence, without any assumptions, the exceptional set in Marstrand's Theorem for planar self-similar sets has zero Hausdorff dimension (rather than just zero Lebesgue measure).

Part (i) of Theorem \ref{thm:prof-planar-sss} was proved by Peres and the author \cite{PeresShmerkin09}, with a different proof yielding many generalizations obtained later in  \cite{HochmanShmerkin12}. We discuss a different approach in Section \ref{sec:irrational-dim}.

Claims (ii) and (iii) are consequences of some deep recent results of  M. Hochman \cite{Hochman14}. We present their proof, modulo a major result from \cite{Hochman14}, in Section \ref{sec:Hochman}.

The last part, concerning positive Lebesgue measure, was recently obtained by the author and B. Solomyak \cite{ShmerkinSolomyak14}. We will outline the proof in Section \ref{sec:absolute-continuity}.

We will also discuss variants valid for (some) self-similar measures, which in most cases are a necessary step towards the proof of the set statements.

In the algebraic, rational type case, the set of exceptional directions can sometimes be  explicitly determined. In particular, this is the case for the one-dimensional Sierpi\'{n}ski gasket, resolving a conjecture of Furstenberg. See Section \ref{subsec:Furstenberg} below.

We comment on the related natural question of what is the Hausdorff \textit{measure} of $\Pi_\alpha A$ in its dimension. When $\hdim A>1$, a partial answer is provided by Theorem \ref{thm:prof-planar-sss}(iv). When $\hdim A\le 1$, in the irrational type case the answer is zero for all $\alpha$. This was proved by Ero\u{g}lu \cite{Eroglu07} under the   OSC (his result predates \ref{thm:prof-planar-sss}(i); he actually proved that the $\hdim(A)$-Hausdorff measure is zero), and recently extended to the general case by Farkas \cite{Farkas14}.

\section{Dimension and projection theorems for measures}

\subsection{Dimensions of measures}

Even if one is ultimately interested in sets, the most powerful methods for studying dimensions of projections involve measures in a natural way. Since a given set may support many dynamically relevant measures (such as self-similar or Gibbs measures) it is also useful to investigate measures for their own sake.

For sets, in this article we focus mostly on Hausdorff dimension. For measures, there are many notions of dimension which are useful or tractable, depending on the problem under consideration. We quickly review the ones we will need. From now on, by a measure we always mean a Radon measure (that is, locally finite and Borel regular) on some Euclidean space $\R^d$.

Given $x\in\supp\mu$, we define the \textbf{lower and upper local dimensions} of $\mu$ at $x$ as
\begin{align*}
\ldim(\mu,x) &= \liminf_{r\searrow 0}\frac{\log\mu B(x,r)}{\log r},\\
\udim(\mu,x) &= \limsup_{r\searrow 0}\frac{\log\mu B(x,r)}{\log r}.
\end{align*}
If $\ldim(\mu,x)=\udim(\mu,x)$, we write $\dim(\mu,x)$ for the common value and call it \emph{the} local dimension at $x$. Local dimensions are functions; in order to obtain a global quantity, one may look at the $\mu$-essential supremum or infimum of the local dimension. This yields four different notions of dimension, out of which the following two are most relevant for studying the dimension of projections:
\begin{align*}
\dim_*\mu &= \sup\{ s: \ldim(\mu,x)\ge s \quad\text{ for }\mu\text{-almost all } x\}\\
\Dim^*\mu &= \inf\{s:\udim(\mu,x)\le s  \quad\text{ for }\mu\text{-almost all } x\}.
\end{align*}
Note that the supremum and infimum in question are attained. In the literature, $\dim_*$ and $\Dim^*$ are known as the \emph{lower Hausdorff dimension} and \emph{(upper) packing dimension} of a measure, respectively. The terminology stems from the following alternative characterization, which is closely related to the mass distribution principle:
\begin{align*}
\dim_*\mu &= \inf\{\hdim A: \mu(A)>0\},\\
\Dim^*\mu &= \sup\{\pdim A: \mu(\R^d\setminus A)=0\}.
\end{align*}
Here $\pdim$ denotes packing dimension. The measure $\mu$ is called \textbf{exact dimensional} if $\dim_*\mu=\Dim^*\mu$ or, alternatively, if $\dim(\mu,x)$ exists and is $\mu$-a.e. constant. Many dynamically defined measures are exact dimensional, but we note that, in general, a fixed projection of an exact dimensional measure needs not be exact dimensional.

A rather different notion of dimension (or rather, a one parameter family of dimensions) is related to the scaling law of the moments of the measure. Namely, given $q\ge 0, q\neq 1$, write
\begin{align}
I_q(\mu,r) &= \int \mu(B(x,r))^{q-1} \,d\mu(x), \label{eq:def-Iqr}\\
D_q(\mu) &= \liminf_{r\searrow 0} \frac{\log I_{q}(\mu,r)}{(q-1)\log r}.\nonumber
\end{align}
The numbers $D_q$ are known as the $L^q$ dimensions of the measure, and are an essential ingredient of the multifractal formalism. The function $q\mapsto D_q\mu$ is always non-increasing, and
\[
D_q(\mu) \le \dim_*\mu \quad\text{for all } q>1.
\]
We refer to \cite{FLR02} for the proof of these facts, as well as further background on the different notions of dimension of a measure and their relationships. We finish by remarking that the value $q=2$ is particularly significant, and $D_2(\mu)$ is also known as the \textbf{correlation dimension} of $\mu$.

\subsection{Dimensions of self-similar measures} \label{subsec-dim-ssm}

If $\mu$ is a measure on $\R^d$ and $g:\R^d\to\R^k$ is a map, we denote the push-forward of $\mu$ under $g$ by $g\mu$, that is, $g\mu(B)=\mu(g^{-1}B)$ for all Borel sets $B$. If $\mathcal{F}=(f_i)_{i\in\Lambda}$ is an IFS and $p=(p_i)_{i\in\Lambda}$ is a probability vector, then there is a unique Borel probability measure $\mu=\mu(\mathcal{F},p)$ such that
\[
\mu = \sum_{i\in\Lambda} p_i\, f_i\mu.
\]
The measure $\mu$ is called the \textbf{self-similar measure} associated to the IFS $\mathcal{F}$ and the weight $p$. For convenience we always assume that $p_i>0$ for all $i$ (otherwise one may pass to the IFS formed by the maps $(f_i:p_i>0)$). In this case, the topological support of $\mu$ is the self-similar set associated to $\F$.

Self-similar measures are always exact dimensional; this is a rather deep fact which (at least in some special cases) can be traced back to ideas of Ledrappier and Furstenberg; see \cite{FengHu09} for a detailed proof. As is the case for sets, dimensions of self-similar measures are well understood under the open set condition. In this case, one has
\begin{equation} \label{eq:sim-dim-measure}
\dim\mu = \frac{\sum_{i\in\Lambda} p_i \log p_i}{\sum_{i\in\Lambda} p_i \log \rho(f_i)}.
\end{equation}
This is an instance of the heuristic formula ``dimension=entropy/Lyapunov exponent'', which often holds for measures invariant under some kind of conformal dynamics.

Regarding $L^q$ dimensions, under the OSC it holds that $D_q\mu=\tau(q)/(q-1)$, where $\tau(q)$ is the only real solution to
\begin{equation*}
\sum_{i\in\Lambda} p_i^q \rho(f_i)^{-\tau(q)}=1.
\end{equation*}
In the special case where $p_i=\rho(f_i)^s$ (where $s=\sdim(\F)$), it can be easily checked that $\dim\mu=D_q\mu=s$ for all $q>0$. These are called the \textbf{natural weights}.

Just as for sets, the formulae for $\dim\mu$ and $D_q\mu$ given above are expected to ``typically'' hold even in the presence of overlaps. For this reason, we call the right-hand side of \eqref{eq:sim-dim-measure} the \textbf{similarity dimension} of $\mu$, and denote it $\sdim\mu$.

\subsection{Projection theorems for measures}

Theorem \ref{thm:Marstrand} has an analog for various notions of dimension of a measure. The standard potential-theoretic proof of Marstrand's Theorem (due to Kaufman) immediately yields a projection theorem for the correlation dimension. Projection theorems for other notions of dimension of a measure were obtained by  Hu and Taylor \cite[Theorem 6.1]{HuTaylor94}, and Hunt and Kaloshin \cite[Theorem 1.1]{HuntKaloshin97}; we remark that they are still fairly straightforward deductions from the proof of Theorem \ref{thm:Marstrand} as presented in i.e. \cite[Chapter 9]{Mattila95}.

\begin{theorem} \label{thm:Marstrand-measures}
Let $\dim$ denote one of $\dim_*$ or $D_q$ where $q\in (1,2]$,  and let $\mu$ be a measure on $\R^2$. The following holds for almost all $\alpha$:
\begin{enumerate}
\item If $\dim\mu\le 1$, then $\dim \Pi_\alpha\mu =\dim \mu$
\item If $\dim\mu>1$, then $\Pi_\alpha\mu$ is absolutely continuous.
\end{enumerate}
\end{theorem}
The theorem fails for $\Dim^*$ and for $D_q$ if $q\notin (1,2]$, see \cite{HuntKaloshin97}. When $D_q\mu>1$, it can be shown that in fact $\Pi_\alpha\mu$ has an $L^q$ density. There is an analogous result valid in higher dimensions.

\section{The irrational case: dimension of projections}
\label{sec:irrational-dim}

\subsection{Projections of some self-similar measures} \label{subsec:proj-irrat-measure}

In this section we discuss the main ideas behind a proof of Theorem \ref{thm:prof-planar-sss}(i). The proof we sketch is based on ideas from \cite{NPS12}, and is a particular case of more general results in \cite{GSSY14}.

For the time being we assume that $f_i(x) = \lambda R_\theta x+t_i$, $i\in\Lambda$, for some $\lambda\in (0,1)$, $\theta\in[0,\pi)$ with $\theta/\pi\notin\mathbb{Q}$, and $t_i\in\R^2$ are translations. In this case, we say that the IFS $(f_i)_{i\in\Lambda}$ is \textbf{homogeneous}. In other words, in a homogeneous IFS, the linear parts are the same for all maps.  Fix a probability vector $(p_i)_{i\in\Lambda}$, and let $\mu$ be the corresponding self-similar measure. The key to our proof of  Theorem \ref{thm:prof-planar-sss}(i) is the following result.

\begin{theorem} \label{thm:L^q-dim-proj}
If $\mu$ is as above, then for any $q\in (1,2]$ and any $\alpha\in [0,2\pi)$,
\[
D_q(\Pi_\alpha \mu)=  \min(D_q\mu,1).
\]
\end{theorem}

We indicate the main steps in the proof. The first main ingredient is the inequality
\begin{equation} \label{eq:subm}
I_q(\Pi_\alpha\mu,\lambda^{k+\ell}) \le C_q\, I_{q}(\Pi_\alpha \mu,\lambda^k)\, I_{q}(\Pi_{\alpha+k\theta}\mu,\lambda^\ell),
\end{equation}
valid for $q>1$ for some constant $C_q>0$. Recall \eqref{eq:def-Iqr}. This is a consequence of the self-similarity of $\mu$. A result of the same kind, for self-similar (and even self-conformal) measures rather than their projections, was obtained by Peres and Solomyak \cite[Equation (3.2)]{PeresSolomyak00}, and the proof here is similar. The homogeneity of the IFS is key in deriving this inequality.

We can rewrite \eqref{eq:subm} as
\[
\varphi_{k+\ell}(\alpha) \le \varphi_k(\alpha)+\varphi_\ell(T^k\alpha),
\]
where $T$ is the $\theta$ rotation on the circle (identified with $[0,2\pi)$), and
\[
\varphi_k(\alpha)= \log I_{q}(P_\alpha\mu,\lambda^k)+\log C_q.
\]
In other words, $\varphi_k$ is a subadditive cocycle over $T$, which is a uniquely ergodic transformation (this is where the irrationality of $\theta/\pi$ gets used).  A result of Furman \cite[Theorem 1]{Furman97} on subadditive cocycles over uniquely ergodic transformations implies that for \textit{all} $\alpha\in [0,2\pi)$ and almost all $\beta\in [0,2\pi)$,
\[
\liminf_{k\to\infty}\frac{\varphi_k(\alpha)}{k(q-1)\log\lambda} \ge \lim_{k\to\infty}\frac{\varphi_k(\beta)}{k(q-1)\log\lambda}.
\]
The limit in the right-hand side exists and is a.e. constant from general considerations (the subadditive ergodic theorem), but in this case we know it equals $\min(D_q\mu,1)$ by Theorem \ref{thm:Marstrand-measures} (it is easy to see that, in the definition of $D_q$, one can take the limit along the sequence $\lambda^k$). This is the step of the proof that uses that $q\le 2$. It follows that $D_q(\Pi_\alpha\mu) \ge \min(D_q\mu,1)$ for all $\alpha$. The opposite inequality is trivial since $D_q$ does not increase under Lipschitz maps and cannot exceed the dimension of the ambient space. This concludes the sketch of the proof of Theorem \ref{thm:L^q-dim-proj}.

We point out that the analog of Theorem \ref{thm:L^q-dim-proj} holds for arbitrary self-similar measures (of irrational type) in the plane, at the price of replacing $L^q$ dimension by Hausdorff dimension. This is a particular case of \cite[Theorem 1.6]{HochmanShmerkin12}. The problem of whether Theorem \ref{thm:L^q-dim-proj} remains valid in this setting, for any values of $q$, remains open.

\subsection{Conclusion of the proof}

We can now finish the proof of Theorem \ref{thm:prof-planar-sss}(i). If $A$ is a self-similar set for a homogeneous IFS satisfying the open set condition, then we know that the self-similar measure $\mu$ with the natural weights satisfies $D_2\mu=\sdim A=\hdim A$, and hence, by Theorem \ref{thm:L^q-dim-proj},
\[
\hdim(\Pi_\alpha A) \ge D_2\Pi_\alpha\mu = \min(D_2\mu,1) = \min(\hdim(A),1) \ge \hdim(\Pi_\alpha A).
\]
This shows that the claim holds when $A$ has this special structure. To conclude the proof, we show that any self-similar set can be approximated in dimension from inside by such a self-similar set; this essentially goes back to \cite{PeresShmerkin09}. We include the proof, since similar approximation arguments have turned to be useful in a variety of situations, see e.g. \cite[Lemma 3.4]{Orponen12}, \cite[Proposition 1.8]{Farkas14}, and \cite[Section 9]{ShmerkinSuomala14}. Recall that an IFS $(f_i)_{i\in\Lambda}$ with attractor $A$ satisfies the \textbf{strong separation condition (SSC)} if the images $f_i(A)$ are pairwise disjoint (this is stronger than the OSC).

\begin{lemma} \label{lem:approx-from-inside}
Let $A$ be a self-similar set in $\R^2$ with $\hdim A>0$. Then for any $\e>0$ there is a self-similar set $A'$ for a homogeneous IFS satisfying the strong separation condition, such that $A'\subset A$ and $\hdim A' \ge \hdim A-\e$.
\end{lemma}
\begin{proof}
Write $s=\hdim A$. It is classical that self-similar sets can be approximated in dimension from inside by self-similar sets satisfying the SSC. For completeness we sketch the argument: given $\e>0$, we can find $r>0$ arbitrarily small and a disjoint collection of balls $\{B(x_i,r)\}$ centres in $A$, with at least $r^{\e/2-s}$ elements. Since $x_i\in A$, it is easy to see that for each $i$ there is a word $j_i$ such that $f_{j_i}(A)\subset B(x_i,r)$ and $\rho(f_{j_i}) \ge \delta r$, where $\delta$ is a positive constant that depends only on the IFS. Then $(f_{j_i})$ satisfies the strong separation condition, the attractor $A''$ is contained in $A$, and its similarity dimension (equal to $\hdim A''$) can be made larger than $s-\e$ by taking $r$ small enough. Hence, we may and do assume that $A$ itself already satisfies the SSC.

A similarity $f$ on $\R^2$ can be written as $\lambda O R_\theta+t$, where $\lambda\in (0,1)$, $O$ is either the identity or reflection around the $x$-axis, $R_\theta$ is rotation by angle $\theta$, and $t\in\R^2$ is a translation. Let $\text{Rot}$ denote the similarities with $O$ equal to the identity, and let $\text{Ref}$ be the remaining ones. We claim that $A$ can be approximated from inside by the attractor of an IFS with elements in $\text{Rot}$ (that still satisfies the SSC). To see this, assume without loss of generality that $f_1\in\text{Ref}$. Fix a large integer $k$, and consider the IFS
\[
\F'_k = (g\in\F^k\cap\text{Rot}) \cup (f_1 g: g\in\F^k\cap\text{Ref}).
\]
A calculation shows that $\sdim(\F'_k)$ can be made arbitrarily close to $s$ by taking $k$ large enough, so this is the desired IFS.

Thus, we assume $\F$ satisfies the SSC and $f_i(x)=\lambda_i R_{\theta_i}x+t_i$ for suitable $\lambda_i\in (0,1)$, $\theta_i\in [0,2\pi)$ and $t_i\in\R^2$. Because the similarities $\lambda R_\theta$ commute, if we write $\F^k= (f_i(x)=S_i x+ t_{k,i})_{i\in\Lambda^k}$, then $S_i$ is determined by the number of times each index $\ell\in\Lambda$ appears in $i=(i_1,\ldots,i_k)$, whence there are fewer than $k^{|\Lambda|}$ different possibilities for $S_i$. Hence, there is some fixed similarity $S$, such that the IFS $\F'_k=(S_i x + t_{k,i}: S_i=S)$ satisfies $N \rho(S)^s \ge k^{-|\Lambda|}$, where $N$ is the number of maps in $\mathcal{F}'_k$. On the other hand, $\rho(S) \le \lambda_{\max}^k$, where $\lambda_{\max}=\max_{i\in\Lambda} \lambda_i<1$. Hence, if we write $\sdim(\mathcal{F}'_k)=s-\e_k$, then
\[
1 = N \rho(S)^{s-\e_k} \ge k^{-|\Lambda|} \rho(S)^{-\e_k} \ge k^{-|\Lambda|} \lambda_{\max}^{-\e_k k}.
\]
Thus $(1/\lambda_{\max})^{\e_k k} \le k^{|\Lambda|}$, and therefore $\e_k\to 0$ as $k\to\infty$. Since $\mathcal{F}'_k$ is a homogeneous IFS satisfying the SSC, whose attractor is contained in $A$ (as it is derived from $\F^k$ by deleting some maps), this completes the proof.
\end{proof}

\section{The rational rotation case: Hochman's Theorem on super exponential concentration} \label{sec:Hochman}
\subsection{Hochman's Theorem}

In the introduction we briefly discussed the following conjecture:

\begin{conjecture} \label{conj:exact-overlap}
If $A$ is a self-similar set in $\R$ with $\hdim(A)<\min(\sdim(A),1)$, then $A$ has exact overlaps.
\end{conjecture}

Although a full solution to the conjecture seems to be beyond reach of current methods, a major breakthrough was recently achieved by M. Hochman \cite{Hochman14}. Hochman proved a weaker form of the conjecture, which allows to establish the full conjecture in a number of important special cases. In order to state his result, we need some definitions. The \textbf{separation constant} of an IFS $\F=( \lambda_i x + t_i)_{i\in\Lambda}$ on the real line is defined as
\[
\Delta(\F) = \left\{
\begin{array}{ccc}
  \infty & \text{if }  \lambda_i\neq\lambda_j \text{ for all }i\neq j, \\
  \min_{i\neq j}\{ |t_i-t_j|:\lambda_i=\lambda_j\} & \text{ otherwise}
\end{array} \right..
\]
Although $\Delta(\F)$ may be infinite, we will only be interested in $\Delta(\F^k)$ for large values of $k$, and this is always finite (already for $k=2$) due to the commutativity of the contraction ratios. The sequence $k\mapsto \Delta(\F^k)$ is always decreasing. Notice also that there is an exact overlap if and only if $\Delta(\F^k)=0$ for some (and hence all sufficiently large) $k$. On the other hand, by pigeonholing it is easy to see that $\Delta(\F^k)$ decays at least exponentially fast in $k$. We say that $\mathcal{F}$ has \textbf{superexponential concentration of cylinders (SCC)} if $\Delta(\F^k)$ decays at superexponential speed or, in other words, if
\[
\lim_{k\to\infty} \frac{-\log\Delta(\F^k)}{k} = \infty.
\]
We can now state Hochman's Theorem:

\begin{theorem} \label{thm:hochman}
If $\mu=\mu(\F,p)$ is a self-similar measure in $\R$ with $\hdim \mu<\min(\sdim\mu,1)$, then $\F$ has super-exponential concentration of cylinders.

In particular, if $A$ is a self-similar set in $\R$ with $\hdim(A)<\min(\sdim(A),1)$, then (the IFS generating) $A$ has super-exponential concentration of cylinders.
\end{theorem}
The proof of this result combines several major new ideas. A key ingredient is an inverse theorem for the growth of entropy under convolutions, which belongs to the field of additive combinatorics. See the survey \cite{Hochman14b} or the introduction of \cite{Hochman14} for an exposition of the main ideas in the proof.

In the remainder of this section, we explain how to apply Theorem \ref{thm:hochman} to the calculation of the dimension of projections of planar self-similar sets and measures, in the rational rotation case. In turn, this will be a key ingredient for establishing absolute continuity of projections. For many other applications of Theorem \ref{thm:hochman}, see \cite{Orponen13, FraserShmerkin14, Shmerkin14, ShmerkinSolomyak14} in addition to \cite{Hochman14}.

\subsection{Projections of the one-dimensional Sierpi\'{n}ski Gasket, and Theorem \ref{thm:prof-planar-sss}(ii)}
\label{subsec:Furstenberg}

\begin{figure}
   \centering
  \includegraphics[width=0.8\textwidth]{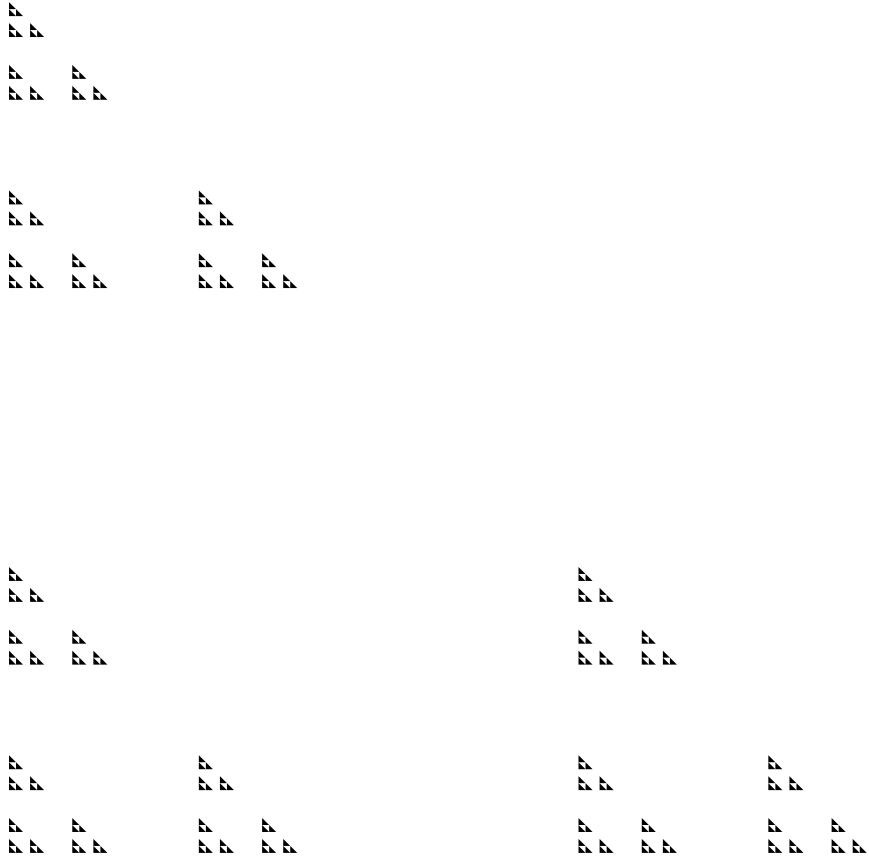}
\caption{The one-dimensional Sierpi\'{n}ski gasket}
\label{fig:gasket}
\end{figure}

The attractor $S$ of the planar IFS $\left( (\tfrac{x}{3},\tfrac{y}{3}), (\tfrac{x+1}{3},\tfrac{y}{3}),(\tfrac{x}{3},\tfrac{y+1}{3})\right)$ is known as the \textbf{one-dimensional Sierpi\'{n}ski Gasket}, see Figure $1$. Since $S$ satisfies the SSC, indeed $\hdim(S)=1$. Because the generating IFS has no rotations, the orthogonal projections of $S$ onto lines are again self-similar sets. Let $P_u(x,y)=x+uy$. Then $S_u:= P_u S$ is homothetic to $\Pi_{\tan^{-1}u} S$. This provides a smooth reparametrization of the orthogonal projections of $S$, which has the advantage that $S_u$ is the attractor of the simpler IFS $\left( \tfrac{x}{3},\tfrac{x+1}{3},\tfrac{x+u}{3}\right)$.

It is clear that $S_u$ has an exact overlap for some values of $u$, for example, for $u=1$. Kenyon showed that there is an exact overlap if and only if $u=p/q$ in lowest terms with $p+q\not\equiv 0\bmod 3$, see \cite[Lemma 6]{Kenyon97}, and provided an expression for the dimension of $S_u$ in this case. Hence, unlike the irrational rotation case, there are exceptional directions, and an infinite number of them. Kenyon rounded off the understanding of projections with rational slope by showing that if $u=p/q$ in lowest terms with $p+q\equiv 0\bmod 3$, then $S_u$ has positive Lebesgue measure (in particular, dimension $1$).

An old (unpublished) conjecture of H. Furstenberg states that $\dim S_u=1$ for all \textit{irrational} $u$. Since $S_u$ has no exact overlap for $u$ irrational, this is a particular case of Conjecture \ref{conj:exact-overlap}. In the same article \cite{Kenyon97}, Kenyon proved that $S_u$ has Lebesgue measure zero for all irrational $u$ (answering a question of Odlyzko), and exhibited a dense $G_\delta$ set of irrational $u$ such that $\hdim S_u=1$. It turns out that a positive solution to Furstenberg's conjecture follows rather easily from Theorem \ref{thm:hochman} (the short deduction is due to Solomyak and the author). The argument is presented in \cite[Theorem 1.6]{Hochman14} just for projections of the one-dimensional Sierpi\'{n}ski gasket. A variant of the proof yields the following more general result.

\begin{theorem} \label{thm:Furstenberg}
Let $\lambda\in (0,1)$ be algebraic, and let $a_i,b_i$ ($i\in\Lambda$) also be algebraic. Suppose that the IFS
\[
\F=\big(\lambda(x+a_i),\lambda(y+b_i)\big)_{i\in\Lambda}
\]
does not have an exact overlap. Write $s=\sdim(\F)=\log |\Lambda|/\log (1/\lambda)$. Let $S$ be the attractor of $\F$, and let $S_u$ be the image of $S$ under $(x,y)\mapsto x+uy$.

Then $\hdim S_u=\min(s,1)$ for all $u$ such that the $S_u$ does not have an exact overlap, and in particular for all but countably many $u$.
\end{theorem}

In the proof we will need the following lemma, see \cite[Lemma 5.10]{Hochman14} for the proof. Given a finite set $B$, the family of polynomial expressions in elements of $B$ of degree at most $k$ will be denoted $\mathcal{P}_k(B)$.
\begin{lemma} \label{lem:polynomial}
Let $B$ be a finite set of algebraic numbers. There is a constant $\delta=\delta(B)>0$ such that if $x\in\mathcal{P}_k(B)$, then either $x=0$ or $|x|\ge \delta^k$.
\end{lemma}

\begin{proof}[Proof of Theorem \ref{thm:Furstenberg}]
The projection $S_u$ is the attractor of
\[
\F_u =\big(\lambda(x+a_i+u b_i)\big)_{i\in\Lambda}.
\]
If $u$ is algebraic, then $\F_u$ is algebraic (i.e. all the parameters are algebraic). As shown in \cite[Theorem 1.5]{Hochman14}, it follows easily from Theorem \ref{thm:hochman} and Lemma \ref{lem:polynomial} that in this case $\hdim S_u=\min(\sdim S_u,1)\ge \min(s,1)$ if $\F_u$ does not have an exact overlap.  Hence from now on we assume that $u$ is transcendental.

The IFS $\F_u^k$ is given by $(\lambda^k x+\sum_{\ell=1}^k \lambda^\ell(a_{i_\ell}+u b_{i_\ell}))_{i\in\Lambda^k}$. In particular, there are $i\neq j\in\Lambda^k$ such that
\[
\Delta(\F_u^k) = x_k +u y_k,
\]
where $x_k=\sum_{\ell=1}^k \lambda^\ell(a_{i_\ell}-a_{j_\ell})$, $y_k=\sum_{\ell=1}^k \lambda^\ell(b_{i_\ell}-b_{j_\ell})$. Since $u$ is transcendental, this can be zero only if $x_k$ and $y_k$ are both zero, but in this case $\F^k$ can be seen to have an exact overlap, which contradicts our hypothesis. Hence for each $k$ either $x_k\neq 0$ or $y_k\neq 0$.

Note that $x_k,y_k\in\mathcal{P}_{k+1}(B)$, where $B=\{ \lambda, a_i-a_j, b_i-b_j:i,j\in\Lambda\}$. Let $\delta=\delta(B)$ be the number given by Lemma \ref{lem:polynomial}. If $x_k=0$ or $y_k=0$, then $|\Delta(\F_u^k)| \ge \min(1,|u|)\delta^{k+1}$. If this happens for infinitely many $k$ then, in light of Theorem \ref{thm:hochman}, $\hdim S_u=\min(\hdim S,1)$ and we are done. Hence we may assume that $x_k y_k\neq 0$ for all $k\ge k_0$.

For any $k\ge k_0$, we hence have
\[
\left|\frac{\Delta(\F_u^k)}{y_k}-\frac{\Delta(\F_u^{k+1})}{y_{k+1}}\right| = \left|\frac{x_k}{y_k}-\frac{x_{k+1}}{y_{k+1}}\right| = \left|\frac{z_k}{y_k y_{k+1}}\right|,
\]
where $z_k\in\cP_{2k+3}(B)$. Therefore, Lemma \ref{lem:polynomial} yields that either $z_k=0$ or $|z_k|\ge \delta^{2k+3}$. Assume first that $z_k=0$ for all sufficiently large $k$, say for all $k\ge k_1\ge k_0$. Then
\[
 |\Delta(F_u^k)| = |y_k (x_{k_1}/y_{k_1}+u)| \ge |x_{k_1}/y_{k_1}+u| \delta^{k+1}\quad\text{for all }k\ge k_1,
\]
so there is no SCC and the conclusion follows again from Theorem \ref{thm:hochman}. Thus, we may and do assume that $z_k\neq 0$ for infinitely many $k$. For any such $k$, since $|y_k|$ is bounded uniformly in $k$,  we conclude that either
\[
|\Delta(\F_u^k)|\ge c \delta^{2k+3}, \quad\text{or}\quad |\Delta(\F_u^{k+1})|\ge c \delta^{2k+3},
\]
for some $c>0$ independent of $k$. This shows that also in this case there is no SCC, so a final application of Theorem \ref{thm:hochman} finishes the proof.
\end{proof}
The same proof works for self-similar measures for $\F_u$.

Part (ii) of Theorem \ref{thm:prof-planar-sss} follows from Theorem \ref{thm:Furstenberg} and Lemma \ref{lem:approx-from-inside}: if $A$ is of irrational type there is nothing to do by part (i). If $A$ is algebraic and of rational type then, given $\e>0$, Lemma \ref{lem:approx-from-inside} provides us with an IFS $\F$ satisfying the hypotheses of Proposition Theorem \ref{thm:Furstenberg} such that $\sdim\F>\hdim A-\e$ (it is clear from the proof that $\F$ is still algebraic, and also has rational rotations, so after iterating we may assume it has no rotations). To finish the proof we apply Theorem \ref{thm:Furstenberg} to $\F$ and let $\e\searrow 0$ along a sequence.

\subsection{Dimension of projections in the rational case}

We now apply Theorem \ref{thm:hochman} to prove Theorem \ref{thm:prof-planar-sss}(iii). Since we already know that in the irrational rotation case there are no exceptional directions at all, it remains to deal with the rational rotation case. Once again, we will first establish a corresponding result for measures, but imposing some additional structure on the IFS. Then we will deduce the general case for sets from Lemma \ref{lem:approx-from-inside}.

\begin{proposition} \label{prop:dim-proj-ssm-rational}
Let $\mathcal{F}=(f_i)_{i\in\Lambda}$ be a planar IFS satisfying the SSC, where $f_i(x)=\lambda x+t_i$ for all $i$, that is, the maps $f_i$ are homotheties with the same contraction ratio. There exists a set $E\subset [0,\pi)$ of zero Hausdorff (and even packing) dimension, such that if $\mu_p=\mu(\F,p)$ denotes the self-similar measure for $\mathcal{F}$ and the weight $p$, then
\[
\hdim\Pi_\alpha\mu_p = \min(\sdim\mu_p,1) \quad\text{for all }\alpha\in [0,\pi)\setminus E.
\]
In particular, if $A$ is the attractor of $\F$, then
\[
\hdim \Pi_\alpha A= \min(\hdim A,1) \quad\text{for all }\alpha\in [0,\pi)\setminus E.
\]
\end{proposition}
\begin{proof}
The latter claim follows by applying the first claim to the natural weights. Once again, instead of working directly with orthogonal projections $\Pi_\alpha$, we work with the family $P_u(x,y)=x+uy$; this makes no difference in the statement since the reparametrization is smooth and hence preserves Hausdorff and packing dimension.

As before, write $\F_u$ for the projected IFS $(\lambda x+P_u t_i)_{i\in\Lambda}$. Note that
\[
\Delta(\F_u^k)= \min_{j\neq j'\in \Lambda^k} \Gamma_{j,j'}(u),
\]
where
\[
\Gamma_{j,j'}(u) = \left|\sum_{i=0}^{k-1} \lambda^i P_u t_{j_i}- \sum_{i=0}^{k-1} \lambda^i P_u t_{j'_i}\right| =\left| P_u\left(\sum_{i=0}^{k-1} \lambda^i( t_{j_i}-t_{j'_i})\right) \right|=: | P_u(t_{j,j'})|.
\]
Since $\F$ satisfies the SSC, $|t_{j,j'}|>c>0$ for some $c=c(\F)$. Let $E$ denote the set of $u$ such that $\F_u$ has SCC. Then, by definition of SCC,
\[
E\subset \bigcap_{\e>0} \bigcup_{K=1}^\infty \bigcap_{k=K}^\infty \bigcup_{j\neq j'\in\Lambda^k} \Gamma_{j,j'}^{-1}(-\e^k,\e^k)=: \bigcap_{\e>0} E_\e.
\]
Fix an interval $I=[-M,M]$. We first note that, since $|t_{j,j'}|>c$, for large enough $k$, the set $I\cap  \Gamma_{j,j'}^{-1}(-\e^k,\e^k)$ can be covered by an interval of length $O(\e^k)$. Hence $\bigcup_{j\neq j'\in\Lambda^k} \Gamma_{j,j'}^{-1}(-\e^k,\e^k)$ can be covered by $|\Lambda|^{2k}$ intervals of length $O(\e^k)$, which implies that
\[
\ubdim\left( \bigcap_{k=K}^\infty \bigcup_{j\neq j'\in\Lambda^k} \Gamma_{j,j'}^{-1}(-\e^k,\e^k)\right) \le O\left(|\log\e|^{-1}\right),
\]
with the implicit constant depending on $|\Lambda|$, where $\ubdim$ is upper box-counting (or Minkowski) dimension. In turn, since packing dimension is $\sigma$-stable and bounded above by $\ubdim$, this shows that $\pdim(E_\e\cap I)=O\left(|\log\e|^{-1}\right)$. Since $I=[-M,M]$ was arbitrary, we conclude that $\pdim(E)=0$. The claim now follows from Theorem \ref{thm:hochman}.
\end{proof}

The above proposition is a particular case of \cite[Theorem 1.8]{Hochman14}, which deals with much more general analytic families of self-similar measures. The proof of the more general result is similar, except that in order to show that $\Gamma_{j,j'}^{-1}(-\e^k,\e^k)$ can be covered efficiently one needs to rely on ``higher-order transversality'' estimates (which are trivial in our setting because $\Gamma_{j,j'}$ is affine).

Claim (iii) of Theorem \ref{thm:prof-planar-sss} follows from Proposition \ref{prop:dim-proj-ssm-rational} and Lemma \ref{lem:approx-from-inside} in exactly the same way as part (ii) followed from Theorem \ref{thm:Furstenberg}.

\section{Absolute continuity of projections}
\label{sec:absolute-continuity}

The methods from \cite{PeresShmerkin09, NPS12, HochmanShmerkin12, Hochman14} that we have discussed so far appear to be intrinsically about dimension (of sets or measures) and so far have not yielded any new information about positive Lebesgue measure or absolute continuity when the similarity dimension exceeds the dimension of the ambient space. Recently, in \cite{Shmerkin14, ShmerkinSolomyak14} these results on dimension have been combined with some new ideas to yield absolute continuity outside a small set of parameters for many parametrized families of self-similar (and related) measures. One particular application of these ideas is the last claim of Theorem \ref{thm:prof-planar-sss}. In this section we discuss the main steps in the proof, referring the reader to \cite{ShmerkinSolomyak14} for the details.

We start by describing the general scheme for proving absolute continuity outside a small set of parameters. The measures $\mu_u$ to which the method applies have an infinite convolution structure: they are the distribution of a random sum
\[
\sum_{n=1}^\infty X_{u,n},
\]
where $X_{u,n}$ are independent Bernoulli random variables, and $\|X_{u,n}\|_\infty$ decreases exponentially uniformly in $u$, so that the series converges absolutely. Once a large integer $k$ is fixed, this allows as to express $\mu_u$ as a convolution $\eta_u*\nu_u$, where $\eta_u$ is the distribution of $\sum_{k|n} X_{u,n}$, and $\nu_u$ is the distribution of $\sum_{k \nmid n} X_{u,n}$. The dimension results discussed in the previous sections can be applied to show that, in many cases, $\nu_u$ has full dimension for all parameters $u$ outside of a small set of exceptions (note that, in the definition of $\nu_u$, we are skipping every $k$-th term only, so $\nu_u$ should be ``almost as large'' as $\mu_u$). On the other hand, adapting a combinatorial method that goes back to Erd\H{o}s \cite{Erdos40} and has become known as the ``Erd\H{o}s-Kahane'' argument, it is often possible to show that the Fourier transform
\[
\widehat{\eta}_u(\xi) = \int \exp(2\pi i \langle x,\xi\rangle )\,d\eta_u(x)
\]
has a power decay, again outside of a small set of possible exceptions (because $\eta_u$ is defined by keeping only every $k$-th term, these measures will have very small dimension and hence a very small power decay, but all that will matter is that it is positive).

Recall that the \textbf{Fourier dimension} of a measure $\eta$ is defined as
\[
\fdim(\eta) = \sup\{ \sigma\ge 0: \exists C, |\widehat{\eta}(\xi)|\le C|\xi|^{-\sigma/2} \}.
\]
Absolute continuity then follows from the following general fact:

\begin{theorem} \label{thm:convolutions}
Let $\eta,\nu$ be Borel probability measures on $\R^d$.
\begin{enumerate}
\item[(i)] If $\hdim\nu+\fdim\eta>d$, then $\eta*\nu$ has an absolutely continuous density.
\item[(ii)] If $D_q\nu+\fdim\eta>d$ for some $q\in (1,2]$, then $\eta*\nu$ has a density in $L^q$.
\end{enumerate}
\end{theorem}
The second part in the case $q=2$ is a rather straightforward consequence of well-known identities relating $D_2$ to energies, and energies to the Fourier transform, while the first part follows from the second for $q=2$ (or any other value of $q$); see \cite[Lemma 2.1]{Shmerkin14}. The second part for arbitrary values of $q\in (1,2]$ is somewhat more involved, and relies on the Littlewood-Paley decomposition; it is proved in \cite[Theorem 4.4]{ShmerkinSolomyak14} (where a version for $q\in (2,+\infty)$ is also established). The intuition behind the theorem is that convolving with a measure of positive Fourier dimension is a smoothing operation (positive Fourier dimension is a kind of ``pseudo-randomness'' indicator), which is enough to ``upgrade'' full or almost full dimension to absolute continuity.

We now indicate how to implement the above strategy for projections of planar self-similar measures $\mu$. We need to assume that the IFS is homogeneous; this is to ensure that the measure we are projecting, and therefore also its projections, have the desired convolution structure.

\begin{theorem} \label{thm:abs-cont-projections}
Let $(f_i(x)=S x+t_i)_{i\in\Lambda}$ be a homogeneous IFS on $\R^2$ satisfying the SSC and $\sdim\F>1$. Then there exists a set $E\subset [0,\pi)$ of zero Hausdorff dimension, such that for all $\alpha\in [0,\pi)\setminus E$ the following holds:
\begin{enumerate}
\item[(i)] Let $\mu_p=\mu(\F,p)$. If $\hdim\mu_p>1$, then $\Pi_\alpha\mu_p$ is absolutely continuous.
\item[(ii)] In the irrational rotation case, if $D_q\mu_p>1$, then $\Pi_\alpha \mu_p$ has an $L^q$ density.
\item[(iii)] Moreover, in the rational rotation case, if $\hdim\mu_p>1$, then $\Pi_\alpha\mu_p$ has a density in $L^q$ for some $q=q(\F,p,\alpha)>1$.
\end{enumerate}
\end{theorem}
Note that, since $\F$ satisfies the SSC, there are explicit formulae for $\hdim\mu, D_q\mu$, and $\lim_{q\to 1^+} D_q\mu=\hdim\mu$, and the set of $p$ to which the theorem applies is nonempty (it includes, for example, the natural weights). In particular, it follows from the second part that, in the irrational rotation case, if $\hdim\mu_p>1$ then $\Pi_\alpha\mu_p$ has an $L^q$ density for an explicit $q>1$ that is independent of $\alpha$ and $p$. Also, if $D_q\mu_p<1$, then $D_q\Pi_\alpha(\mu_p)<1$, and $\Pi_\alpha(\mu_p)$ cannot have an $L^q$ density, so the second part is sharp up to the endpoint. Thus we know a lot less about the density of the projections in the rational rotation case.

\begin{proof}[Sketch of proof]
First of all, by replacing $\F$ by $\F^k$ for suitable $k$, in the rational rotation case we may assume that $S$ is a homothety, i.e. there are no rotations at all.

The self-similar measure $\mu_p$ is the distribution of the random sum $\sum_{n=1}^\infty X_n$, where $\mathbb{P}(X_n=S^n t_i)=p_i$, and the $X_n$ are independent. Indeed, this measure is easily checked to satisfy the defining relation $\mu_p=\sum_{i\in\Lambda} p_i\, f_i\mu_p$. Since $p$ is fixed in the proof we drop any explicit reference to it from now on.

As indicated above, let $k$ be a large integer to be determined later, and let $\eta, \nu$ be the distribution of the random sums $\sum_{k|n} X_{n}$, $\sum_{k\nmid n} X_{n}$ respectively, so that $\mu=\eta*\nu$ and therefore $\Pi_\alpha\mu=\Pi_\alpha\eta *\Pi_\alpha\nu$. This fits with the general description above, since $\Pi_\alpha\mu$ is the distribution of $\sum_{n=1}^\infty X_{\alpha,n}$, where $\mathbb{P}(X_{\alpha,n}=\Pi_\alpha(S^n t_i))=p_i$, and the $X_{\alpha,n}$ are independent, and likewise with $\Pi_\alpha\eta, \Pi_\alpha\nu$.

Both $\eta, \nu$ are again homogeneous self-similar measures of the same rotation type (irrational or no rotation), which also satisfy the SSC. Moreover, a direct calculation shows that
\begin{align*}
\hdim\nu &= (1-1/k)\hdim\mu,\\
D_q\nu &= (1-1/k)D_q\mu.
\end{align*}

Consider first the irrational rotation case, and suppose $p$ is such that $D_q\mu>1$. Provided we chose $k$ large enough, then also $D_q\nu>1$. By Theorem \ref{thm:L^q-dim-proj}, $D_q\Pi_\alpha\nu=1$ for all $\alpha$. On the other hand, a combinatorial argument similar to (although slightly more involved than) the classical Erd\H{o}s' argument from \cite{Erdos40}, shows that $\fdim \Pi_\alpha \eta>0$ outside of a possible exceptional set of zero Hausdorff dimension. See \cite[Proposition 3.3]{ShmerkinSolomyak14} (this also holds in the no rotation case). Claim (ii) then follows from Theorem \ref{thm:convolutions}, and we have already seen that this implies (i) in the irrational rotation case.

The first claim in the no-rotations case follows in the same way, using Proposition \ref{prop:dim-proj-ssm-rational} instead of Theorem \ref{thm:L^q-dim-proj}. A priori this gives no information whatsoever about the densities (the reason being that Theorem \ref{thm:hochman} is about Hausdorff dimension and it is unknown if it holds for $L^q$ dimension for any $q$). However, in \cite[Theorem 5.1]{ShmerkinSolomyak14} we have shown that for any homogeneous self-similar measure $\tau$, and in particular for $\tau=\Pi_\alpha \nu$ in the no-rotations case,
\[
\lim_{q\to 1^+} D_q(\tau) = \hdim\tau.
\]
(This is immediate from the explicit formulae under the OSC, the point is that it holds regardless of overlaps.) Hence, if $\alpha$ is such that $\hdim(\Pi_\alpha\nu)=1$ and $\fdim(\Pi_\alpha\eta)>0$, there is a (non-explicit) $q>1$ such that $D_q(\Pi_\alpha\nu)+\fdim(\Pi_\alpha\eta)>1$. The third claim then follows again from Theorem \ref{thm:convolutions}.
\end{proof}

Using Lemma \ref{lem:approx-from-inside} once again, we conclude the proof of Theorem \ref{thm:prof-planar-sss}(iv) in the by now familiar way.

Unfortunately, the proof of Theorem \ref{thm:abs-cont-projections} (and hence of Theorem \ref{thm:prof-planar-sss}(iv))  is completely non-effective. The reason is that it seems very hard to prove that a given projection of a self-similar measure has power Fourier decay, even though we know that all outside of a zero-dimensional set do!

\section{Further results}
\label{sec:further-results}

We briefly discuss projections of other natural classes of sets and measures. This section has some overlap with \cite[Sections 8 and 9]{FFJ14}.

\subsection{Bernoulli convolutions}

Given $\lambda\in (0,\tfrac12)$, the \textbf{Bernoulli convolution} $\nu_\lambda$ is the self-similar measure for the IFS $(\lambda x-1,\lambda x+1)$ with weights $(\tfrac{1}{2},\tfrac{1}{2})$. Alternatively, $\nu_\lambda$ is the distribution of the random sum $\sum_{n=0}^\infty \pm \lambda^n$, where the signs are chosen independently with equal probabilities; this explains the name. When $\lambda\in (0,1/2]$, the generating IFS satisfies the OSC and the measure $\nu_\lambda$ is well understood; however, for $\lambda\in (\tfrac{1}{2},1)$, surprisingly little is known. It is known since Erd\H{o}s \cite{Erdos39} that if $\lambda^{-1}$ is a Pisot number (an algebraic integer larger than $1$, all of whose algebraic conjugates are smaller than $1$ in modulus), then $\nu_\lambda$ is singular. It is not known if there are any other $\lambda\in (\tfrac12,1)$ for which $\nu_\lambda$ is singular. Solomyak \cite{Solomyak95} proved that $\nu_\lambda$ is absolutely continuous with an $L^2$ density for almost all $\lambda\in (\tfrac12,1)$. This is a kind of Marstrand Theorem for a family of \textit{nonlinear} projections. Using the method described in Section \ref{sec:absolute-continuity}, the author proved in \cite{Shmerkin14} that $\nu_\lambda$ is absolutely continuous for $\lambda$ outside of a zero Hausdorff dimension set of exceptions, and in \cite{ShmerkinSolomyak14} we showed that, furthermore, outside this exceptional set, $\nu_\lambda$ has a density in $L^q$ for some non-explicit $q=q(\lambda)>1$. These results rely heavily on Theorem \ref{thm:hochman}.

\subsection{Self-similar sets in higher dimension}

Much less is known about projections of self-similar sets in higher dimensions. In dimensions $d\ge 3$ there is no neat decomposition into ``rational rotation'' and ``irrational rotation'' cases. In particular, if the orthogonal parts of all the maps in the IFS coincide, then they cannot generate a dense subgroup of the orthogonal group - this is problematic for generalizing Theorems \ref{thm:L^q-dim-proj} and \ref{thm:abs-cont-projections}. Also, the lack of commutativity precludes approximation arguments such as Lemma \ref{lem:approx-from-inside}.  Nevertheless, the more flexible approach of \cite{HochmanShmerkin12} yields the following, see \cite[Theorem 1.6 and Corollary 1.7]{HochmanShmerkin12}.
\begin{theorem}
Let $A\subset\R^d$, $d\ge 2$, be the attractor of the IFS $( \lambda_i O_i x+t_i)_{i\in\Lambda}$, where $\lambda_i\in (0,1)$, $O_i\in\mathbb{O}_d$ and $t_i\in\R^d$. Assume the SSC holds. Fix $1\le k<d$ and let $G_{d,k}$ denote the Grassmanian of $k$-dimensional subspaces of $\R^d$.

Suppose that the action of the semigroup generated by the $O_i$ on $G_{d,k}$ is transitive, that is, $\{ O_{i_1}\cdots O_{i_n}\pi:i_j\in\Lambda,n\in\N \}$ is dense for some (and therefore all) $\pi\in G_{d,k}$. Then for all $C^1$ maps $g:\R^d\to\R^k$ without singular points,
\[
\hdim(g A)=\min(\hdim A,k)
\]
\end{theorem}
Note that in dimension $d=2$, the transitivity condition is met precisely for self-similar sets of irrational type. Once again, this follows from a corresponding result for measures. Using an approximation argument, Farkas \cite[Theorem 1.6]{Farkas14} was able to remove the SSC assumption.

\subsection{Projections of self-affine carpets}

If the maps $f_i$ in an IFS are affine rather than similarities, the attractor is called a \textbf{self-affine set}. Dimension problems for self-affine sets are notoriously difficult, and almost nothing beyond the general results of Marstrand and others is known about their orthogonal projections, outside of some special classes known as \textit{self-affine carpets}. Roughly speaking, a self-affine carpet is the attractor of an IFS of affine maps that map the unit square onto non-overlapping rectangles with some special pattern (generally speaking, it is required that when projecting these rectangles onto either the $x$ or $y$-axes, there are either no overlaps or exact overlaps).

In \cite{FJS10}, it was proved that under a suitable irrationality condition, for many self-affine carpets $A\subset\R^2$ it holds that $\hdim \Pi A=\min(\hdim(A),1)$ for all projections $\Pi$ other than the principal ones (which are always exceptional for carpets). The proof was based on ideas of \cite{PeresShmerkin09} and did not extend to measures. Recently, based on the approach of \cite{Hochman14}, Ferguson, Fraser and Sahlsten \cite{FFS15} obtained the corresponding results for Bernoulli measures for the natural Markov partition for the $(x,y)\mapsto (px, qy)\bmod 1$ toral endomorphism. This was extended to Gibbs measures by Almarza \cite{Almarza14}.

\subsection{Sums of Cantor sets}

The arithmetic sum $A+B$ of two sets $A,B\subset\R^d$ is $\{ x+y: x\in A, y\in B\}$. Up to an homothety, this is the projection of $A\times B$ under a $45$ degree projection. More generally, the family $A+u B$ is a reparametrization of the projections of $A\times B$ (other than the horizontal projection). The methods discussed in the previous sections can be applied (with suitable modifications) to yield the following analog of Theorem \ref{thm:prof-planar-sss}. Following the terminology of \cite{PeresShmerkin09}, we say that the attractors of $(\lambda_i x+t_i)$ and $(\lambda'_j x+t'_j)$ are \textbf{algebraically resonant} if $\log \lambda_i/\log \lambda'_j$ is rational for all $i,j$.

\begin{theorem}  \label{thm:sum-sss}
 Let $A,B\subset\R$ be self-similar sets, and write $s=\hdim(A)+\hdim(B)$.
 \begin{enumerate}
  \item[(i)] If $A$ and $B$ are not algebraically resonant, then $\hdim(A+uB)=\min(s,1)$ for all $u\in\R\setminus\{0\}$.
  \item[(ii)] If both $A$ and $B$ are given by algebraic parameters, then $\{ u:\hdim(A+uB)<\min(s,1)\}$ is countable.
  \item[(iii)] Without any assumptions, if $s\le 1$, then $\hdim\{ u:\hdim(A+uB)<s\}=0$.
  \item[(iv)] If $s>1$, then $\hdim\{ u: \mathcal{L}(A+uB)=0\}=0$.
 \end{enumerate}
\end{theorem}
Part (i) was proved in \cite{PeresShmerkin09}, (ii) and (iii) follow in a very similar manner to the corresponding claims in Theorem \ref{thm:prof-planar-sss}, and the last claim is \cite[Theorem E]{ShmerkinSolomyak14}. Measure versions of these results exist, see \cite[Theorem 1.4]{HochmanShmerkin12} and \cite[Theorem D]{ShmerkinSolomyak14}.

Much is known for sumsets beyond self-similar sets and measures. One of the first results in the area, due to Moreira, was a version of Theorem \ref{thm:sum-sss}(i) for attractors of \textit{nonlinear} IFSs (under standard regularity assumptions). See \cite{Moreira98}. A general version of this (valid also for Gibbs measures) was obtained in \cite[Theorem 1.4]{HochmanShmerkin12}. Furthermore, it follows from \cite[Theorem 1.3]{HochmanShmerkin12} that if $A,B\subset [0,1]$ are closed and invariant under $x\mapsto px \bmod 1$, $x\mapsto qx\bmod 1$ respectively, with $\log p/\log q\notin\mathbb{Q}$, then
\[
\hdim(A+uB)=\min(\hdim(A)+\hdim(B),1)\quad\text{ for all }u\in\R\setminus\{0\}.
\]
For $u=1$, this was another conjecture of Furstenberg.

\subsection{Projections of random sets and measures}

There is a vast, and growing, literature on geometric properties of random sets and measures of Cantor type, including the behavior of their projections. We do no more than indicate some references for further reading. Generally speaking, there are two main strands of research in this area. One concerns random sets and measures that include deterministic ones as a special case. In this direction, Falconer and Jin \cite{FalconerJin14, FalconerJin14b} investigated projections of random cascade measures (and related models) on self-similar sets, obtaining generalizations of several of the results we discussed. In \cite{FalconerJin14b}, these results were applied in a clever way to study the dimension of linear sections of \textit{deterministic} self-similar sets. The second line of research concerns sets and measures with a large degree of spatial independence; one of the most popular such models is fractal percolation, consisting in iteratively selecting random squares in the dyadic (or $M$-adic) grid. With stronger independence, one can typically say a lot more, for example proving positive Lebesgue measure, and even nonempty interior, for all projections simultaneously (compare with Theorem \ref{thm:prof-planar-sss}(iii)). See e.g. \cite{RamsSimon14} for results of this type for fractal percolation. A general approach to the study of measures with strong spatial independence was recently developed in \cite{ShmerkinSuomala14}. We refer the reader to this paper for many further references and detailed statements.

\bibliographystyle{plain}
\bibliography{projections_survey}

\end{document}